\documentclass[10pt]{amsart}
\usepackage{amsmath}
\usepackage{amssymb}
\usepackage{amsthm}
\usepackage{indentfirst}
\usepackage[latin1]{inputenc}
\usepackage{geometry}
\usepackage{amsfonts}
\usepackage{enumitem}
\usepackage{cancel}
\usepackage[bookmarks=false,colorlinks=true,linkcolor=black,citecolor=black,filecolor=black,urlcolor=black]{hyperref} 
\usepackage{verbatim,listings}
\DeclareSymbolFont{largesymbols}{OMX}{yhex}{m}{n}
\DeclareMathAccent{\widehat}{\mathord}{largesymbols}{"62}

\newcommand{\MC}{G} 

\newcommand{\Q}{{\mathbb Q}}
\newcommand{\Z}{{\mathbb Z}}
\newcommand{\F}{{\mathbb F}}
\newcommand{\Gal}{\textnormal{Gal}}
\newcommand{\vect}{\textnormal{vect}}
\newcommand{\Cen}{\textnormal{Cen}}
\newcommand{\kar}{\textnormal{char}}
\newcommand{\Ker}{{\textnormal{Ker\,}}}
\newcommand{\Aut}{\textnormal{Aut}}
\newcommand{\tr}{{\textnormal{tr}}}
\newcommand{\End}{\textnormal{End}}
\newcommand{\inv}{{^{-1}}}
\newcommand{\GEN}[1]{\left\langle #1 \right\rangle}
\newcommand{\U}{\mathcal{U}}

\theoremstyle{plain}
\newtheorem{theorem}{Theorem}[section]

\newtheorem{corollary}[theorem]{Corollary}

\theoremstyle{definition}
\newtheorem{remark}[theorem]{Remark}

\newtheorem{example}[theorem]{Example}

\begin{document}

\title[Construction of minimal non-abelian left group codes]{Construction of minimal non-abelian left group codes}

\author{Gabriela Olteanu} 
\address{Department of Statistics-Forecasts-Mathematics, Babe\c s-Bolyai University,
Str. T. Mihali 58-60, 400591 Cluj-Napoca, Romania}
\email{gabriela.olteanu@econ.ubbcluj.ro}

\author{Inneke Van Gelder}
\address{Department of Mathematics, Vrije Universiteit Brussel,
Pleinlaan 2, 1050 Brussels, Belgium}
\email{ivgelder@vub.ac.be}

\date{\today}

\thanks{The research is supported by the grant PN-II-RU-TE-2009-1 project ID\_303, the grant PN-II-ID-PCE-2012-4-0100 and by the Research
Foundation Flanders (FWO - Vlaanderen).}

\begin{abstract}
Algorithms to construct minimal left group codes are provided. These are based on results describing a complete set of orthogonal primitive idempotents in each Wedderburn component of a semisimple finite group algebra $\F G$ for a large class of groups $G$.

As an illustration of our methods, alternative constructions to some best linear codes over $\F_2$ and $\F_3$ are given. Furthermore, we give constructions of non-abelian left group codes.
\keywords{left group codes \and linear codes \and primitive idempotents \and Wedderburn decomposition \and finite group algebras}

\end{abstract}

\maketitle

\section{Introduction}

In this paper $\F=\F_s$, the finite field with $s$ elements. A linear code over $\F$ of length $n$ and rank $k$ is a linear subspace $C$
with dimension $k$ of the vector space $\F^n$. The standard basis of $\F^n$ is denoted by $E=\{e_1,\ldots,e_n\}$. The vectors in $C$ are called
codewords, the size of a code is the number of codewords and equals $s^k$. The distance of a code is the minimum distance between distinct codewords, i.e. the number of elements in which they differ. The weight of a codeword is the distance to the zero codeword. The weight distribution is listing for each integer $i$ the number of codewords of weight $i$. A linear code of length $n$, dimension $k$, and distance $d$ is called a $[n,k,d]$-code. Bounds on the minimum distance of linear codes are known, see \cite{GUAVA,1998Brouwer,2006Grassl,2007Grassl}. A linear code $C$ can be represented as the $\F^n$-span of a minimal set of codewords, these basis codewords are often collated in the rows of a matrix known as a generating matrix for the code $C$.

For any group $G$, we denote by $\F G$ the group algebra over $G$ with coefficients in $\F$. If $G$ is a group of order $n$ and 
$C\subseteq \F^n$ is a linear code, then we say that $C$ is a left $G$-code (respectively a $G$-code) if there is a bijection $\phi:E\rightarrow G$
such that the linear extension of $\phi$ to an isomorphism $\phi:\F^n\rightarrow \F G$ maps $C$ to a left ideal (respectively a two-sided ideal) of
$\F G$. A left group code (respectively a group code) is a linear code which is a left $G$-code (respectively a $G$-code) for some group $G$. A (left)
cyclic group code (respectively, abelian, metacyclic, nilpotent group code, \ldots) is a linear code which is a (left) $G$-code for some cyclic group
(respectively, abelian, metacyclic, nilpotent group, \ldots) $G$. The underlying group is not uniquely determined by the code itself. That means that it is possible that a (left) non-abelian group code can also be realized as an abelian group code.

In \cite{2009BernalRioSimon,sabin} it is proved that if $C$ is a two-sided metacyclic group code then $C$ is an abelian group code. However an example of a two-sided group code which is not an abelian group codes was recently given in \cite{Pillado2013}. For left group codes, Bernal, del R\'io and Sim\'on proved that for every non-abelian group $G$ and every prime $p$ not dividing the order of $G$ there is a left $G$-code over some field of characteristic $p$ which is not an abelian group code \cite[Proposition 3.3]{2009BernalRioSimon}. 
Since it is more likely to find left group codes which are not abelian group codes, we will study (minimal) left group codes. 

For a metacyclic group $G=\GEN{a,b\mid a^m=1=b^n, ba=a^rb}$ where $\gcd(m,r)=1$, $r^n\equiv 1 \mod m$ and both $m$ and $n$ are odd, Sabin and Lomonaco \cite{sabin} gave an algorithm to determine minimal left codes in $\F G$ where $\F$ has characteristic 2. They discovered several good metacyclic codes and they expressed the hope (also inspired by results on other non-abelian codes \cite{1989ChengSloane}) that more ``good'' and perhaps even ``best'' codes may be discovered among the ideals of non-abelian group rings. As Sabin and Lomonaco did, we obtain an algorithm, but for a larger class of groups and fields, and rediscover some best codes. We also included an implementation of our algorithm in the GAP package Wedderga \cite{Wedderga}.

\section{Preliminaries}\label{pre}

When $R$ is a semisimple ring (i.e. $R$ is a direct sum of a finite number of minimal left ideals), then every left ideal $L$ of $R$ is of the form $L=Re$, where $e$ is an idempotent of $R$. Therefore, we can use the idempotents to characterize the decompositions of semisimple rings as direct sums of minimal left ideals. In particular, let $R=\oplus_{i=1}^t L_i$ be a decomposition of a semisimple ring as direct sums of minimal left ideals. Then, there exists a family $\{e_1,\dots,e_t\}$ of elements of $R$ such that: each $e_i$ is a non-zero idempotent element, if $i\neq j$, then $e_ie_j=0$, $1=e_1+\cdots+e_t$ and each $e_i$ cannot be written as $e_i=e_i'+e_i''$, where $e_i',e_i''$ are idempotents such that $e_i',e_i''\neq 0$ and $e_i'e_i''=0$, $1\leq i\leq t$. Conversely, if there exists a family of idempotents $\{e_1,\dots,e_t\}$ satisfying the previous conditions, then the left ideals $L_i=Re_i$ are minimal and $R=\oplus_{i=1}^t L_i$. Such a set of idempotents is called a complete set of orthogonal primitive idempotents of the ring $R$. Note that such a set is not uniquely determined. When studying left group codes, it is useful to study minimal left group codes, i.e. codes associated with minimal left ideals, and in particular primitive idempotents of finite group algebras.

Recall that, given a decomposition of a semisimple ring $R$ as direct sums of minimal left ideals, we can group isomorphic left ideals
together. 
The sum of all left ideals isomorphic to one in the decomposition, turns out to be a minimal two-sided ideal of $R$ which is simple as a ring. Also
the decomposition of $R$ as direct sums of two-sided ideals is related to a family of idempotents. Let $R=\oplus_{i=1}^s A_i$ be a decomposition of a
semisimple ring as direct sums of minimal two-sided ideals. Then, there exists a family $\{e_1,\dots,e_s\}$ of elements of $R$ such that: each
$e_i$ is a non-zero central idempotent element, if $i\neq j$, then $e_ie_j=0$, $1=e_1+\cdots+e_s$ and each $e_i$ cannot be written as $e_i=e_i'+e_i''$,
where $e_i',e_i''$ are central idempotents such that $e_i',e_i''\neq 0$ and $e_i'e_i''=0$, $1\leq i\leq s$. The elements $\{e_1,\ldots,e_s\}$ are
called the primitive central idempotents of $R$ and they give rise to the well-known Wedderburn-Artin Theorem. Using this knowledge backwards, it can
be helpful to consider the Wedderburn decomposition and the primitive central idempotents in order to determine a complete set of orthogonal primitive
idempotents.

From now on, $G$ denotes an arbitrary finite group such that $\F G$ is semisimple. By Maschke's Theorem this is equivalent to saying that the
order of $G$ is coprime to the characteristic of $\F$. The notation $H\leq G$ (resp. $H\unlhd G$) means that $H$ is a subgroup (resp. normal subgroup)
of $G$. For $H\leq G$, $g\in G$ and $h\in H$, we define $H^g=g^{-1}Hg$ and $h^g=g^{-1}hg$. Analogously, for $\alpha \in \F G$ and $g\in G$,
$\alpha^g=g^{-1}\alpha g$. For $H\leq G$, $N_G(H)$ denotes the normalizer of $H$ in $G$ and we set $\widetilde{H}=|H|^{-1}\sum_{h\in H} h$, an
idempotent of $F G$, and if $H=\langle g\rangle$ then we simply write $\widetilde{g}$ for $\widetilde{\langle g\rangle}$. 

The classical method for computing primitive central idempotents in a semisimple group algebra $\F G$ involves characters of the group $G$. All the characters of any finite group are assumed to be characters in $\overline{\F}$, a fixed algebraic closure of the field $\F$. For an irreducible character $\chi$ of $G$, $e(\chi)=\frac{\chi(1)}{|G|}\sum_{g\in G}\chi(g^{-1})g$ is the primitive central idempotent of $\overline{\F}G$ associated to $\chi$ and $e_{\F}(\chi)$ is the only primitive central idempotent $e$ of $\F G$ such that $\chi(e)\neq 0$. The field of character values of $\chi$ over $\F$ is defined as $\F(\chi)=\F(\chi(g) : g\in G)$, that is the field extension of $\F$ generated over $\F$ by the image of $\chi$. The automorphism group $\Aut(\overline{\F})$ acts on $\overline{\F}G$ by acting on the coefficients, that is $\sigma\sum_{g\in G} a_gg=\sum_{g\in G}\sigma(a_g)g$, for $\sigma\in\Aut(\overline{\F})$ and $a_g\in\overline{\F}$. Following \cite{Yamada1973}, we know that $e_{\F}(\chi)=\sum_{\sigma\in\Gal(\F(\chi)/\F)}\sigma e(\chi)$. 

New methods for the computation of the primitive central idempotents in a group algebra do not involve characters. The main ingredient in this theory is the following element, introduced in \cite{Jespers2003}. If $K \unlhd H\leq G$, then let $\varepsilon(H,K)$ be the element of $\Q H\subseteq \Q G$ defined as 
\begin{eqnarray*}
\varepsilon(H,K)&=&
\left\{\begin{array}{ll}
\widetilde{K} & \mbox{if } H=K, \\
\prod_{M/K\in \mathcal{M}(H/K)}(\widetilde{K}-\widetilde{M}) & \mbox{if } H\neq K,
\end{array}\right.
\end{eqnarray*}
where $\mathcal{M}(H/K)$ denotes the set of minimal normal non-trivial subgroups of $H/K$. Furthermore, $e(G,H,K)$ denotes the sum of the different $G$-conjugates of $\varepsilon(H,K)$. By \cite[Theorem 4.4]{Olivieri2004}, the elements $\varepsilon(H,K)$ are the building blocks for the primitive central idempotents of $\Q G$ for abelian-by-supersolvable groups $G$.

We introduce some notations and results from \cite{Broche2007}. Let $\F$ and $G$ be as before, with $|G|=n$. Throughout the paper, we fix an 
algebraic closure of $\F$, denoted by $\overline{\F}$. For every positive integer $k$ coprime with $s$, $\xi_k$ denotes a primitive $k$-th root of
unity in $\overline{\F}$ and $o_k(s)$ denotes the multiplicative order of $s$ modulo $k$. Recall that $\F(\xi_k)\simeq\F_{s^{o_k(s)}}$, the field of
order $s^{o_k(s)}$. Let $\mathcal{Q}$ denote the subgroup of $\Z_n^*$, the group of units of the ring $\Z_n$, generated by the class of $s$ and
consider $\mathcal{Q}$ acting on $G$ by $t\cdot g=g^t$. The $s$-cyclotomic classes of $G$ are the orbits of $G$ under the action of $\mathcal{Q}$ on
$G$. For a cyclic group $A$, let $A^*$ be the group of irreducible characters in $\overline{\F}$ of $A$ and let $\mathcal{C}(A)$ denote the set of $s$-cyclotomic classes of $A^*$, which consist of linear faithful characters of $A$. 

Let $K\unlhd H\leq G$ be such that $H/K$ is cyclic of order $k$ and $C\in\mathcal{C}(H/K)$. If $\chi\in C$ and $\tr=\tr_{\F(\xi_k)/\F}$ denotes the field trace of the Galois extension $\F(\xi_k)/\F$, then we set $$\varepsilon_C(H,K)=|H|^{-1}\sum_{h\in H} \tr(\chi(hK))h^{-1}=[H:K]^{-1}\widetilde{K}\sum_{X\in H/K}\tr(\chi(X))h_X^{-1} ,$$ where $h_X$ denotes a representative of $X\in H/K$. Note that $\varepsilon_C(H,K)$ does not depend on the choice of $\chi\in C$. Furthermore, $e_C(G,H,K)$ denotes the sum of the different $G$-conjugates of $\varepsilon_C(H,K)$. Note that the elements $\varepsilon_C(H,K)$ will occur in Theorem~\ref{mainfinite} as the building blocks for the primitive central idempotents of finite group algebras.

If $H$ is a subgroup of $G$, $\psi$ a linear character of $H$ and $g\in G$, then $\psi^g$ denotes the character of $H^g$ given by $\psi^g(h^g)=\psi(h)$. This defines an action of $G$ on the set of linear characters of subgroups of $G$. Note that if $K=\Ker\psi$, then $\Ker\psi^g=K^g$ and therefore the rule $\psi\mapsto\psi^g$ defines a bijection between the set of linear characters of $H$ with kernel $K$ and the set of linear characters of $H^g$ with kernel $K^g$. This bijection maps $s$-cyclotomic classes to $s$-cyclotomic classes and hence induces a bijection $\mathcal{C}(H/K)\rightarrow\mathcal{C}(H^g/K^g)$. 

Let $K\unlhd H\leq G$ be such that $H/K$ is cyclic. Then the action from the previous paragraph induces an action of $N=N_G(H)\cap N_G(K)$ on $\mathcal{C}(H/K)$ and it is easy to see that the stabilizer of a cyclotomic class in $\mathcal{C}(H/K)$ is independent of the cyclotomic class. We denote by $E_G(H/K)$ the stabilizer of such (and thus of any) cyclotomic class in $\mathcal{C}(H/K)$ under this action.

A strong Shoda pair of $G$ is a pair $(H,K)$ of subgroups of $G$ satisfying the following conditions:
\begin{itemize}
\item[(SS1)] $K\leq H\unlhd N_G(K)$,
\item[(SS2)] $H/K$ is cyclic and a maximal abelian subgroup of $N_G(K)/K$, and
\item[(SS3)] for every $g\in G\setminus N_G(K)$, $\varepsilon(H,K)\varepsilon(H,K)^g=0$.
\end{itemize}

It is also proven in \cite{Broche2007} that $\Cen_G(\varepsilon_C(H,K))=E_G(H/K)$ in the case when $(H,K)$ is a strong Shoda pair of $G$. The following Theorem gives a description of the primitive central idempotents of $\F G$ given by strong Shoda pairs and the associated simple components.

\begin{theorem}\cite[Theorem 7]{Broche2007}\label{mainfinite}
Let $G$ be a finite group and $\F$ a finite field of order $s$ such that $\F G$ is semisimple. Let $(H,K)$ be a strong Shoda pair of $G$ and $C\in\mathcal{C}(H/K)$. Then $e_C(G,H,K)$ is a primitive central idempotent of $\F G$ and $$\F G e_C(G,H,K)\simeq M_{[G:H]}(\F_{s^{o/[E:H]}}),$$ where $E=E_G(H/K)$ and $o$ is the multiplicative order of $s$ modulo $[H:K]$.
\end{theorem}

\begin{remark}\rm From \cite[Theorem 7]{Broche2007}, we also know that there is a strong relation between the primitive central idempotents in a rational group algebra $\Q G$ and the primitive central idempotents in a finite group algebra $\F G$ that makes use of the strong Shoda pairs of $G$. More precisely, if $X$ is a set of strong Shoda pairs of $G$ and every primitive central idempotent of $\Q G$ is of the form $e(G,H,K)$ for $(H,K)\in X$, then every primitive central idempotent of $\F G$ is of the form $e_C(G,H,K)$ for $(H,K)\in X$ and $C\in\mathcal{C}(H/K)$.
\end{remark}

Let $\chi$ be an irreducible (complex) character of $G$. Then $\chi$ is strongly monomial if there is a strong Shoda pair $(H,K)$ of $G$
and a linear character $\theta$ of $H$ with kernel $K$ such that $\chi=\theta^G$, the induced character of $G$. The group $G$ is strongly monomial if
every irreducible character of $G$ is strongly monomial.

A complete description of the primitive central idempotents and the simple components for strongly monomial groups is given in \cite{Broche2007}.

\begin{corollary}\label{SSP}
If $G$ is a strongly monomial group and $\F$ is a finite field of order $s$ such that $\F G$ is semisimple, then every primitive central idempotent of $\F G$ is of the form $e_C(G,H,K)$ for $(H,K)$ a  strong Shoda pair of $G$ and $C\in\mathcal{C}(H/K)$. Furthermore, for every strong Shoda pair $(H,K)$ of $G$ and every $C\in\mathcal{C}(H/K)$, $$\F Ge_C(G,H,K)\simeq M_{[G:H]}(\F_{s^{o/[E:H]}}),$$ where $E=E_G(H/K)$ and $o$ is the multiplicative order of $s$ modulo $[H:K]$.
\end{corollary}

However, in some cases we have more information on the algebra isomorphism given in the previous Theorem. We can express the simple algebra $\F
Ge_C(G,H,K)$ in terms of a crossed product.

If $R$ is a (not necessarily finite) unital associative ring and $G$ is a group then $R*^{\alpha}_{\tau} G$ denotes a crossed product with 
action $\alpha:G\rightarrow \Aut(R)$ and twisting (a two-cocycle) $\tau:G\times G \rightarrow \U(R)$ (see for example \cite{Passman1989}), i.e.
$R*^{\alpha}_{\tau} G$ is the associative ring $\bigoplus_{g\in G} R u_g$ with multiplication given by the following rules: $u_g a = \alpha_g(a) u_g$
and $u_g u_h = \tau(g,h) u_{gh}$, for $a\in R$ and $g,h\in G$. Recall that a classical crossed
product is a crossed product $L*^{\alpha}_{\tau} G$, where $L/F$ is a finite Galois extension (of not necessarily finite fields), $G = \Gal(L/F)$ is
the Galois group of the field extension $L/F$ and $\alpha$ is the natural action of $G$ on $L$. A classical crossed product $L *^{\alpha}_{\tau} G$ is
denoted by $(L/F,\tau)$ \cite{Reiner1975}. If the twisting $\tau$ is cohomologicaly trivial, then the classical crossed product is isomorphic to a
matrix algebra over its center. Moreover, when $\tau=1$ we get an explicit isomorphism. More precisely, denoting the matrix associated to an
endomorphism $f$ in a basis $B$ as $[f]_B$, we have the following result.

\begin{theorem}\label{reiner}\cite[Corollary 29.8]{Reiner1975}
Let $L/F$ be a finite Galois extension and $n=[L:F]$. The classical crossed product $(L/F,1)$ is isomorphic (as $F$-algebra) to $M_n(F)$. Moreover, an isomorphism is given by
$$\begin{array}{rcccc}
\psi:(L/F,1) & \longrightarrow & \End_F(L) & \longrightarrow  & M_n(F)   \\
xu_{\sigma} & \longmapsto & x'\circ \sigma & \longmapsto & [x'\circ \sigma]_B,
\end{array}$$
for $x\in L$, $\sigma\in \Gal(L/F)$, $B$ an $F$-basis of $L$ and where $x'$ denotes multiplication by $x$ on $L$.
\end{theorem}

Let $(H,K)$ be a strong Shoda pair of a group $G$, $C\in\mathcal{C}(H/K)$, $E=E_G(H/K)$ and $\phi:E/H\rightarrow E/K$ a left inverse of the canonical projection $E/K\rightarrow E/H$.
As mentioned in the proofs of \cite{Broche2007}, with ideas from \cite{Olivieri2004}, we know that $\F E\varepsilon_C(H,K)=\F H\varepsilon_C(H,K)*^{\alpha}_{\tau}E/H = \F(\zeta_{[H:K]})*^{\alpha}_{\tau}E/H$ and the action and twisting are given by 
\begin{eqnarray*}
\alpha_{gH}(\zeta_{[H:K]}) &=& \zeta_{[H:K]}^i, \mbox{ if } yK^{\phi(gH)}=y^iK \mbox{ and}\\
\tau(gH,g'H) &=& \zeta_{[H:K]}^j, \mbox{ if }  \phi(gg'H)\inv\phi(gH)\phi(g'H)=y^jK,
\end{eqnarray*}
for $gH,g'H\in E/H$ and integers $i$ and $j$. 
Since the action $\alpha$ is faithful, $\F(\zeta_{[H:K]})*^{\alpha}_{\tau}E/H$ can be described as a classical crossed product $(\F(\zeta_{[H:K]})/Z,\tau)$, where $Z=\F_{s^{o/[E:H]}}$ is the center of the algebra, which is determined by the Galois action $\alpha$. Hence $E/H\simeq \Gal(\F(\zeta_{[H:K]})/Z)$.
If moreover the twisting $\tau$ is trivial, we know a concrete isomorphism $\F Ge_C(G,H,K)\simeq M_{[G:E]}(\F E\varepsilon_C(H,K))\simeq M_{[G:H]}(\F_{s^{o/[E:H]}})$.

Using the description of the primitive central idempotents and the Wedderburn components of a semisimple finite group algebra $\F G$, we were able to describe a complete set of orthogonal primitive idempotents of $\F G$ in the case where $G$ is nilpotent \cite{2011vangelder}. This description will be used in section~\ref{applications} to construct minimal left nilpotent group codes.
\begin{theorem}\cite[Theorem 3.3]{2011vangelder}\label{nilpotent}
Let $\F$ be a finite field and $G$ a finite nilpotent group such that $\F G$ is semisimple. Let $(H,K)$ be a strong Shoda pair of $G$, $C\in\mathcal{C}(H/K)$ and set $e_C=e_C(G,H,K)$, $\varepsilon_C=\varepsilon_C(H,K)$, $H/K=\langle\overline{a}\rangle$, $E=E_G(H/K)$. Let $E_2/K$ and $H_2/K=\langle\overline{a_2}\rangle$ (respectively $E_{2'}/K$ and $H_{2'}/K=\langle\overline{a_{2'}}\rangle$) denote the $2$-parts (respectively $2'$-parts) of $E/K$ and $H/K$ respectively. Then $\langle\overline{a_{2'}}\rangle$ has a cyclic complement $\langle\overline{b_{2'}}\rangle$ in $E_{2'}/K$.

A complete set of orthogonal primitive idempotents of $\F Ge_C$ consists of the conjugates of $\beta_{e_C}=\widetilde{b_{2'}}\beta_2\varepsilon_C$ by the elements of $T_{e_C}=T_{2'}T_2T_E$, where $T_{2'}=\{1,a_{2'},a_{2'}^2,\dots,a_{2'}^{[E_{2'}:H_{2'}]-1}\}$, $T_E$ denotes a right transversal of $E$ in $G$ and $\beta_2$ and $T_2$ are given according to the cases below.

\begin{enumerate}
\item If $H_2/K$ has a complement $M_2/K$ in $E_2/K$ then $\beta_2=\widetilde{M_2}$. Moreover, if $M_2/K$ is cyclic, then there exists $b_2\in E_2$ such that $E_2/K$ is given by the following presentation $$\langle \overline{a_2},\overline{b_2}\mid \overline{a_2}\hspace{1pt}^{2^n}=\overline{b_2}\hspace{1pt}^{2^k}=1, \overline{a_2}\hspace{1pt}^{\overline{b_2}}=\overline{a_2}\hspace{1pt}^r \rangle,$$ and if $M_2/K$ is not cyclic, then there exist $b_2,c_2\in E_2$ such that $E_2/K$ is given by the following presentation $$\langle \overline{a_2},\overline{b_2},\overline{c_2}\mid \overline{a_2}\hspace{1pt}^{2^n}=\overline{b_2}\hspace{1pt}^{2^k}=\overline{c_2}\hspace{1pt}^2=1, \overline{a_2}\hspace{1pt}^{\overline{b_2}}=\overline{a_2}\hspace{1pt}^r, \overline{a_2}\hspace{1pt}^{\overline{c_2}}=\overline{a_2}\hspace{1pt}^{-1}, [\overline{b_2},\overline{c_2}]=1 \rangle,$$ with $r\equiv 1 \mod 4$ (or equivalently $\overline{a_2}\hspace{1pt}^{2^{n-2}}$ is central in $E_2/K$). Then
\begin{enumerate}
\item\label{fid1i} $T_2=\{1,a_2,a_2^2,\dots, a_2^{2^k-1}\}$, if $\overline{a_2}\hspace{1pt}^{2^{n-2}}$ is central in $E_2/K$ (unless $n\leq 1$) and $M_2/K$ is cyclic; and
\item\label{fid1ii}  $T_2=\{1,a_2,a_2^2,\dots,a_2^{d/2-1},a_2^{2^{n-2}},a_2^{2^{n-2}+1},\dots,a_2^{2^{n-2}+d/2-1}\}$, where $d=[E_2:H_2]$, otherwise. 
\end{enumerate}
\item\label{fid2}  If $H_2/K$ has no complement in $E_2/K$, then there exist $b_2,c_2\in E_2$ such that $E_2/K$ is given by the following presentation
\begin{eqnarray*}
\langle \overline{a_2},\overline{b_2},\overline{c_2}&\mid& \overline{a_2}\hspace{1pt}^{2^n}=\overline{b_2}\hspace{1pt}^{2^k}=1, \overline{c_2}\hspace{1pt}^2=\overline{a_2}\hspace{1pt}^{2^{n-1}}, \overline{a_2}\hspace{1pt}^{\overline{b_2}}=\overline{a_2}\hspace{1pt}^r \overline{a_2}\hspace{1pt}^{\overline{c_2}}=\overline{a_2}\hspace{1pt}^{-1},[\overline{b_2},\overline{c_2}]=1 \rangle, \end{eqnarray*}
with $r\equiv 1 \mod 4$. In this case, $\beta_2=\widetilde{b_2}\frac{1+xa_2^{2^{n-2}}+ya_2^{2^{n-2}}c_2}{2}$ and $$T_2=\{1,a_2,a_2^2,\dots, a_2^{2^k-1},c_2,c_2a_2,c_2a_2^2,\dots,c_2a_2^{2^k-1}\},$$ with $x,y\in\F$, satisfying $x^2+y^2=-1$ and $y\neq 0$.
\end{enumerate}
\end{theorem}

This theorem provided a straightforward implementation in GAP. Nevertheless, in case (~\ref{fid2}), there might occur some difficulties
finding solutions for the equation $x^2+y^2=-1$ for $x,y\in\F$ and $y\neq 0$. However, we were able to overcome this problem (\cite[Remark 3.4]{2011vangelder}). Computations involving strong
Shoda pairs and primitive central idempotents were already provided in the GAP package Wedderga \cite{Wedderga} and we've
included our new algorithms there. 

\section{A complete set of orthogonal primitive idempotents in $\F G$}

Throughout this section we will assume that $\F$ is a finite field of order $s$ and $G$ is a finite group such that the order of $G$ is coprime to $s$. We will focus on simple components of $\F G$ which are determined by a strong Shoda pair $(H,K)$ and a class $C\in\mathcal{C}(H/K)$ such that $\tau(gH,g'H)=1$ for all $g,g'\in E=E_G(H/K)$ (with notation as in section~\ref{pre}). For such a component, we describe a complete set of orthogonal primitive idempotents. This construction is based on the isomorphism of Theorem~\ref{reiner} on classical crossed products with trivial twisting. Such a description, together with the description of the primitive central idempotent $e_C=e_C(G,H,K)$ determining the simple component, yields a complete set of irreducible modules and will be applied in section~\ref{applications} to construct codes.

Before we do so, we need a basis of $\F(\zeta_{[H:K]})/\F(\zeta_{[H:K]})^{E/H}=\F_{s^o}/\F_{s^{o/[E:H]}}$ (with $o$ the multiplicative order of 
$s$ modulo $[H:K]$) of the form $\{w^x\mid x\in E/H\}$ with $w\in \F(\zeta_{[H:K]})$. That such a basis exists follows from the well-known Normal
Basis Theorem which states that if $K/F$ is a finite Galois extension, then there exists an element $w\in K$ such that $\{\sigma(w)\mid \sigma\in
\Gal(K/F)\}$ is an $F$-basis of $K$, a so-called normal basis, whence $w$ is called normal in $K/F$. Recall that $E/H$, the Galois group of $\F_{s^o}$
over $\F_{s^{o/[E:H]}}$, is cyclic and generated by the Frobenius automorphism $x\mapsto x^{s^{o/[E:H]}}$ (see \cite{Roman2006}). Hence if $\beta\in
\F_{s^o}$ is such that the $[E:H]$ elements $\{ \beta, \beta^{s^{o/[E:H]}}, \dots, \beta^{(s^{o/[E:H]})^{{[E:H]}-1}} \} $ are linearly independent,
then this set forms a normal basis for $\F_{s^o}$ over $\F_{s^{o/[E:H]}}$. For a background on the construction of normal bases, see Artin
\cite{1973Artin}, L\"uneburg \cite{1985Luneburg}, Lenstra \cite{1991Lenstra} and Gao \cite{1993Gao}. The construction of normal bases is implemented
in GAP in the method \verb NormalBase .

Now we can state our main result on primitive idempotents.

\begin{theorem}\label{idempotents}
Let $G$ be a finite group and $\F$ a finite field of order $s$ such that $s$ is coprime to the order of $G$. Let $(H,K)$ be a strong Shoda pair of $G$ such that $\tau(gH,g'H)=1$ for all $g,g'\in E=E_G(H/K)$, and let $C\in\mathcal{C}(H/K)$. Let $\varepsilon=\varepsilon_C(H,K)$ and $e=e_C(G,H,K)$. Let $w$ be a normal element of $\F_{s^o}/\F_{s^{o/[E:H]}}$ (with $o$ the multiplicative order of $s$ modulo $[H:K]$) and $B$ the normal basis determined by $w$. Let $\psi$ be the isomorphism between $\F E \varepsilon$ and the matrix algebra $M_{[E:H]}(\F_{s^{o/[E:H]}})$ with respect to the basis $B$ as stated in Theorem~\ref{reiner}.
Let $P,A\in M_{[E:H]}(\F_{s^{o/[E:H]}})$ be the matrices $$P= \left( \begin{array}{rrrrrr}
1 & 1 & 1 & \cdots & 1 & 1\\
1 & -1 & 0 & \cdots & 0 & 0\\
1 & 0 & -1 & \cdots & 0 & 0\\
\vdots & \vdots & \vdots & \ddots & \vdots & \vdots\\
1 & 0 & 0 & \cdots & -1 & 0\\
1 & 0 & 0 & \cdots & 0 & -1\\
\end{array} \right)
\quad \text{and} \quad 
A= \left( \begin{array}{ccccc}
0 & 0 & \cdots & 0 & 1\\
1 & 0 & \cdots & 0 & 0\\
0 & 1 & \cdots & 0 & 0\\
\vdots & \vdots & \ddots & \vdots & \vdots\\
0 & 0 & \cdots & 0 & 0\\
0 & 0 & \cdots & 1 & 0\\
\end{array} \right).$$
Then $$\{x\widetilde{T_1}\varepsilon x^{-1} \mid x\in T_2\GEN{x_e}\}$$ is a complete set of orthogonal primitive idempotents of $\F G e$ where $x_e=\psi^{-1}(PAP^{-1})$, $T_1$ is a transversal of $H$ in $E$ and $T_2$ is a right transversal of $E$ in $G$. By $\widetilde{T_1}$ we denote the element $\frac{1}{|T_1|}\sum_{t\in T_1}{t}$ in $\F G$. 
\end{theorem}
\begin{proof}
Consider the simple component $$\F Ge \simeq M_{[G:E]}(\F E \varepsilon)\simeq M_{[G:H]}( \F_{s^{o/[E:H]}})$$ of $\F G$. Without loss of generality we may assume that $G=E$. Indeed, if we obtain a complete set of orthogonal primitive idempotents of $\F E\varepsilon$, then the conjugates by the transversal $T_2$ of $E$ in $G$ will give a complete set of orthogonal primitive idempotents of $\F Ge$ since $e=\sum_{t\in T_2}\varepsilon^t$ and different $\varepsilon^t$'s are orthogonal.

From now on we assume that $G=E$ and $e=\varepsilon$ and denote $n=[E:H]$. Then $B=\{w^{gH} : g\in T_1\}$. Since $G/H$ acts on $\F He$ via the 
induced conjugation action on $H/K$, it is easily seen that the action of $G/H$ on $B$ is regular. Hence it is readily verified that for each
$g\in T_1$, $\psi(ge)$ is a permutation matrix, and
$$\psi(\widetilde{T_1}e)=\frac{1}{n}\left( \begin{array}{ccccc}
1 & 1 & \cdots & 1 & 1\\
1 & 1 & \cdots & 1 & 1\\
1 & 1 & \cdots & 1 & 1\\
\vdots & \vdots & \ddots & \vdots & \vdots\\
1 & 1 & \cdots & 1 & 1\\
1 & 1 & \cdots & 1 & 1\\
\end{array} \right).$$
Clearly $\psi(\widetilde{T_1}e)$ has eigenvalues $1$ and $0$, with respective eigenspaces $V_1=\vect\{(1,1,\dots,1)\}$ and $V_0=\vect\{(1,-1,0,\dots,0),(1,0,-1,\dots,0),\dots,(1,0,0,\dots,-1)\}$, where $\vect(S)$ denotes the vector space generated by the set $S$. Hence $$\psi(\widetilde{T_1}e)=PE_{11}P^{-1},$$ where we denote by $E_{ij}\in M_{n}(\F_{s^{o/[E:H]}})$ the matrices whose entries are all 0 except in the $(i,j)$-spot, where it is 1. One knows that $\{E_{11},E_{22},\dots,E_{nn}\}$ and hence also $$\{\psi(\widetilde{T_1}e)=PE_{11}P^{-1},PE_{22}P^{-1},\dots,PE_{nn}P^{-1}\}$$ forms a complete set of orthogonal primitive idempotents of $M_{n}(\F_{s^{o/[E:H]}})$. Let $y=\psi(x_e)=PAP^{-1}$. As $$E_{22}=AE_{11}A^{-1}, \dots, E_{nn}=A^{n-1}E_{11}A^{-n+1}$$ we obtain that $$\{\psi(\widetilde{T_1}e),y\psi(\widetilde{T_1}e)y^{-1}, \dots, y^{n-1}\psi(\widetilde{T_1}e)y^{-n+1}\}$$ forms a complete set of orthogonal primitive idempotents of $M_{n}(\F_{s^{o/[E:H]}})$. Hence, applying $\psi^{-1}$ gives us a complete set of orthogonal primitive idempotents of $\F G e$.
\end{proof}

This method yields a detailed description of a complete set of orthogonal primitive idempotents of $\F G$ when $G$ is a strongly monomial group 
such that there exists a complete and non-redundant set of strong Shoda pairs $(H,K)$ satisfying $\tau(gH,g'H)=1$ for all $g,g'\in E_G(H/K)$. Remark
that similar techniques are used in \cite{2013JdROVG} to construct a complete set of orthogonal primitive idempotents of the rational group
algebra $\Q G$ with $G$ as before. For example, the symmetric group $S_4$ and the alternating group $A_4$ of degree 4 have a trivial twisting in all
Wedderburn components of their group rings. Trivially, all abelian groups are included and it is also easy to prove that for all dihedral groups
$D_{2n}=\GEN{a,b\mid a^n=b^2=1,\ a^b=a^{-1}}$ there exists a complete and non-redundant set of strong Shoda pairs with trivial twisting since the
group action involved has order 2 and hence is faithful. On the other hand, for quaternion groups $Q_{4n}=\GEN{x,y \mid x^{2n} = y^4 = 1,\ x^n = y^2,\
x^y = x^{-1}}$, one can verify that the strong Shoda pair $(\GEN{x},1)$ yields a non-trivial twisting.

Even when the group is not strongly monomial or some strong Shoda pairs yield a non-trivial twisting, our description of primitive idempotents can still be used in the components determined by a strong Shoda pair with trivial twisting. This implies that we can always compute some minimal left $G$-codes over a finite field $\F$ for a finite group $G$ of order coprime to $\kar(\F)$.

\section{A class of metacyclic groups}

In this section, we show that our main result can be applied to the metacyclic groups of the form $C_{q^m}\rtimes C_{p^n}$ with $C_{p^n}$ acting faithfully on $C_{q^m}$ and $p$ and $q$ different primes, and fields $\F$ of size $s$ coprime to $pq$. 

Throughout this section $p$ and $q$ are different primes, $m$ and $n$ are positive integers and $G=\GEN{a}\rtimes \GEN{b}$ with $|a|=q^m$, $|b|=p^n$ 
and $\GEN{b}$ acts faithfully on $\GEN{a}$ (i.e. the centralizer of $a$ in $\GEN{b}$ is trivial). Let $\sigma$ be the automorphism of $\GEN{a}$ given by $\sigma(a)=a^b$ and
assume that $\sigma(a)=a^r$ with $r\in \Z$. As the kernel of the restriction map $\Aut(\GEN{a})\rightarrow \Aut\left(\GEN{a^{q^{m-1}}}\right)$ has
order $q^{m-1}$, it intersects $\GEN{\sigma}$ trivially and therefore the restriction of $\sigma$ to $\GEN{a^{q^{m-1}}}$ also has order $p^n$. This
implies that $q\equiv 1 \mod p^n$ and thus $q$ is odd. Therefore, $\Aut\left(\GEN{a^{q^j}}\right)$ is cyclic for every $j=0,1,\dots,m$ and
$\GEN{\sigma}$ is the unique subgroup of $\Aut(\GEN{a})$ of order $p^n$. So, for every $i=1,\dots,m$, the image of $r$ in $\Z/q^i \Z$ generates the
unique subgroup of $\U(\Z/q^i \Z)$ of order $p^n$. In particular, $r^{p^n}\equiv 1 \mod q^m$ and $r^{p^j}\not\equiv 1 \mod q$ for every
$j=0,\dots,n-1$. Therefore, $r\not\equiv 1 \mod q$ and hence $\MC'=\GEN{a^{r-1}}=\GEN{a}$. 

In \cite{Olivieri2004} more information was obtained on the strong Shoda pairs needed to describe the primitive central idempotents of the rational (and hence of a semisimple finite) group algebra of a finite metabelian group. We recall the statement.

\begin{theorem}\cite[Theorem 4.7]{Olivieri2004}\label{SSPmetabelian}
Let $G$ be a finite metabelian group and let $A$ be a maximal abelian subgroup of $G$ containing the commutator subgroup $G'$. The primitive central idempotents of $\Q G$ are the elements of the form $e(G,H,K)$, where $(H,K)$ is a pair of subgroups of $G$ satisfying the following conditions:
\begin{enumerate}
\item \label{metabelian1}$H$ is a maximal element in the set $\{B\leq G \mid A\leq B \mbox{ and } B'\leq K\leq B\}$;
\item \label{metabelian2}$H/K$ is cyclic.
\end{enumerate}
\end{theorem}

Using this description of the strong Shoda pairs, we get a complete and non-redundant set of strong Shoda pairs of $\MC$ consisting of two types:
\begin{enumerate}[label=\rm(\roman{*}), ref=\roman{*}]
\item \label{SP1} $\left(\MC,L_i:=\GEN{a,b^{p^i}}\right), \; i=0,\dots,n$,\\
\item \label{SP2} $\left(\GEN{a},K_j:=\GEN{a^{q^j}}\right), \; j=1,\dots,m$.
\end{enumerate}

It is easy to verify that for these strong Shoda pairs the corresponding twisting is trivial. Hence we can describe a complete set of orthogonal primitive idempotents in each simple component of $\F \MC$ using Theorem~\ref{idempotents}.

\section{Examples of minimal left group codes}\label{applications}

In this section we will provide some illustrative examples of minimal left group codes making use of the computation of primitive idempotents. 
For these examples we used the computer algebra system GAP \cite{GAP} and the packages GUAVA \cite{GUAVA} and Wedderga \cite{Wedderga}. The
implementation of the used methods (based on Theorem~\ref{nilpotent} and Theorem~\ref{idempotents}) is now included in Wedderga. 

Note that each element $c$ in $\F G$ is of the form $c=\sum_{i=1}^n f_i g_i$, where we fix an ordering $\{g_1,g_2,\dots,g_n \}$ of the group elements
of $G$ and $f_i\in \F$. If we look at $c$ as a codeword, we will rather write $[f_1 f_2 \ldots f_n]$.

\begin{example}
We consider the finite group algebra $\F_2G$ over the nilpotent metacyclic group $$G=\GEN{a,b\mid a^9=1,b^3=1,ba=a^4b}$$ and fix an ordering $S$ of $G$.
 \begin{lstlisting}[frame=trbl]
gap> A:=FreeGroup("a","b");;a:=A.1;;b:=A.2;;
gap> G:=A/[a^9,b^3,b*a*b^(-1)*a^(-4)];;
gap> F:=GF(2);;
gap> FG:=GroupRing(F,G);;
gap> S:=AsSet(G);;
 \end{lstlisting}

Using Theorem~\ref{SSPmetabelian}, we see that in the Wedderburn decomposition of $\F_2G$ only one simple component ($M_3(\F_4)$) can possibly contribute to a non-abelian left group code, given by the strong Shoda pair $(H,K)=(\GEN{a},1)$. To define the primitive central idempotent of $\F_2G$ associated to this simple component, we have to define an $s$-cyclotomic class of irreducible $\overline{\F_2}$-characters of $H/K$, which consists of linear faithful characters. All these linear faithful characters are defined by sending the generator of $H/K$ to a power (coprime to $[H:K]$) of a fixed primitive $[H:K]$-root of unity. Using the generator of $H/K$, such a cyclotomic class can be represented by an $s$-cyclotomic class modulo $[H:K]$, which consists of integers coprime to $[H:K]$. With this information we can compute a complete set of orthogonal primitive idempotents in the simple
component $\F_2Ge_C(G,H,K)$.

\begin{lstlisting}[frame=trbl]
gap> H:=Subgroup(G,[G.1]);;
gap> K:=Subgroup(G,[]);;
gap> N:=Normalizer(G,K);;
gap> epi:=NaturalHomomorphismByNormalSubgroup(N,K);;
gap> QHK:=Image(epi,H);;
gap> gq:=MinimalGeneratingSet(QHK)[1];;
gap> C:=CyclotomicClasses(Size(F),Index(H,K))[2];;
gap> P:=PrimitiveIdempotentsNilpotent(FG,H,K,C,[epi,gq]);;
\end{lstlisting}

Using the first primitive idempotent $e$, we can consider the left ideal $\F_2Ge$ of $\F_2G$ and compute its corresponding code.
\begin{lstlisting}[frame=trbl]
gap> e:=P[1];;
gap> CodeWordByGroupRingElement(F,S,e);
[ 0*Z(2), 0*Z(2), 0*Z(2), 0*Z(2), 0*Z(2), 0*Z(2), 0*Z(2), 0*Z(2), 0*Z(2), 
  0*Z(2), 0*Z(2), 0*Z(2), 0*Z(2), 0*Z(2), 0*Z(2), 0*Z(2), 0*Z(2), 0*Z(2), 
  0*Z(2), 0*Z(2), 0*Z(2), Z(2)^0, 0*Z(2), 0*Z(2), 0*Z(2), 0*Z(2), Z(2)^0 ]
gap> Ge := List(G,g->g*e);; time;
4
gap> B := List(Ge,x->CodeWordByGroupRingElement(F,S,x));;
gap> code := GeneratorMatCode(B,F);
a linear [27,18,1..2]3..9 code defined by generator matrix over GF(2)
gap> MinimumDistance(code);
2
gap> DisplayProfile();
  count  self/ms  chld/ms  stor/kb  chld/kb  package  function                     
     28       32        0     9692        0  (oprt.)  CodeWordByGroupRingElement   
      1      152        0    17038        0  (oprt.)  PrimitiveIdempotentsNilpotent
              12               234                    OTHER                        
             196             26966                    TOTAL 
\end{lstlisting}
In this way we have constructed a linear $[27,18,2]$-code over $\F_2$ by means of a minimal left nilpotent metacyclic group code.
We remark that starting with a different strong Shoda pair, determining the same Wedderburn component, can yield another code with different parameters. For example when taking the strong Shoda pair $(\GEN{b,a^3},\GEN{b})$, the obtained code is a linear $[27,6,6]$-code over $\F_2$.
Although the above code is constructed by a metacyclic group, it can also be obtained as a $C_{27}$-group code.

The time consuming parts in this example are the computations of the idempotents and of the translation of the group ring elements into code words, as displayed by the function \verb+DisplayProfile+.
\end{example}

The following is an example of a left abelian-by-metacyclic group code which turns out to be a best linear code (as one can check using \cite{GUAVA}, \cite{1998Brouwer} or \cite{2007Grassl}), i.e. a code which reaches the maximum bound on the minimum distance. 
This allows an alternative construction of a linear $[105,3,60]$-code over $\F_2$.
\begin{example}
We consider the group ring $\F_2G$ over the group $$G=\GEN{a,b\mid a^7=1,b^3=1,c^5=1,ba=a^4b,[a,c]=1,[b,c]=1}.$$
\begin{lstlisting}[frame=trbl]
gap> A:=FreeGroup("a","b","c");;a:=A.1;;b:=A.2;;c:=A.3;;
gap> G:=A/[a^7,b^3,c^5,b*a*b^(-1)*a^(-4),c*a*c^(-1)*a^(-1),c*b*c^(-1)*b^(-1)];; 
gap> F:=GF(2);;
gap> FG:=GroupRing(F,G);;
gap> S:=AsSet(G);;      
gap> H:=StrongShodaPairs(G)[5][1];;
gap> K:=StrongShodaPairs(G)[5][2];;
gap> N:=Normalizer(G,K);;
gap> epi:=NaturalHomomorphismByNormalSubgroup(N,K);;
gap> QHK:=Image(epi,H);;
gap> gq:=MinimalGeneratingSet(QHK)[1];;
gap> C:=CyclotomicClasses(Size(F),Index(H,K))[2];;
gap> P:=PrimitiveIdempotentsTrivialTwisting(FG,H,K,C,[epi,gq]);;
gap> e:=P[1];;        
gap> Ge := List(G,g->g*e);; time;
3984
gap> B := List(Ge,x->CodeWordByGroupRingElement(F,S,x));;
gap> code := GeneratorMatCode(B,F);
a linear [105,3,1..60]51..52 code defined by generator matrix over GF(2)
gap> MinimumDistance(code);
60
gap> LowerBoundMinimumDistance(105,3,2);
60
gap> UpperBoundMinimumDistance(105,3,2);
60
gap> DisplayProfile();
  count  self/ms  chld/ms  stor/kb  chld/kb  package  function                     
    105    23793        0  7412626        0  (oprt.)  CodeWordByGroupRingElement   
      1    36299        0  118184*        0  (oprt.)  PrimitiveIdempotentsTrivial*
               4               229                    OTHER                        
           60096           192313*                    TOTAL 
\end{lstlisting}
In this way we have constructed a best $[105,3,60]$-code by means of a minimal left abelian-by-metacyclic group code. Although it is unclear whether this code can be realized by an abelian group or not. We found a $C_{105}$-group code with the same weight distribution as in the example. Since the high length we were not able to determine whether these codes are permutation equivalent or not.
\end{example}

The following example is one of a left metacyclic group code. This allows an alternative construction of a linear $[20,4,12]$-code over $\F_3$.
\begin{example}
We consider the group ring $\F_3G$ over the group $$G=\GEN{a,b\mid a^5=1,b^4=1,ba=a^2b}.$$
\begin{lstlisting}[frame=trbl]
gap> A:=FreeGroup("a","b");;a:=A.1;;b:=A.2;;
gap> G:=A/[a^5,b^4,b*a*b^(-1)*a^(-2)];;
gap> F:=GF(3);;         
gap> FG:=GroupRing(F,G);;
gap> S:=AsSet(G);;
gap> H:=StrongShodaPairs(G)[4][1];;
gap> K:=StrongShodaPairs(G)[4][2];; 
gap> N:=Normalizer(G,K);;
gap> epi:=NaturalHomomorphismByNormalSubgroup(N,K);;
gap> QHK:=Image(epi,H);;
gap> gq:=MinimalGeneratingSet(QHK)[1];;
gap> C:=CyclotomicClasses(Size(F),Index(H,K))[2];;
gap> P:=PrimitiveIdempotentsTrivialTwisting(FG,H,K,C,[epi,gq]);;
gap> e1:=P[3];;   
gap> Ge1 := List(G,g->g*e1);; time;
160
gap> B1 := List(Ge1,x->CodeWordByGroupRingElement(F,S,x));;
gap> code1 := GeneratorMatCode(B1,F);
a linear [20,4,1..8]8..13 code defined by generator matrix over GF(3)
gap> MinimumDistance(code1);
8
gap> e2:=P[2];;   
gap> Ge2 := List(G,g->g*e2);; time;
72
gap> B2 := List(Ge2,x->CodeWordByGroupRingElement(F,S,x));;
gap> code2 := GeneratorMatCode(B2,F);
a linear [20,4,1..12]8..13 code defined by generator matrix over GF(3)
gap> MinimumDistance(code2);
12
gap> LowerBoundMinimumDistance(20,4,3);
12
gap> UpperBoundMinimumDistance(20,4,3);
12
gap> DisplayProfile();
  count  self/ms  chld/ms  stor/kb  chld/kb  package  function                     
     40      360        0   112511        0  (oprt.)  CodeWordByGroupRingElement   
      1     1504        0   430439        0  (oprt.)  PrimitiveIdempotentsTrivial*
               8                88                    OTHER                        
            1872            543040                    TOTAL 
\end{lstlisting}
In this way we have constructed a $[20,4,8]$-code and a best $[20,4,12]$-code by means of minimal left metacyclic group codes. Notice that the choice of the primitive idempotent is crucial to obtain a best code.
We also checked that the $[20,4,8]$-code cannot be realized by an abelian group. However we found a $C_{20}$-code with the same weight distribution as the $[20,4,12]$-code, but we were not able to determine whether these codes are permutation equivalent or not.
\end{example}

The following example is one of a left metacyclic group code over $\F_2$ which is not an abelian group code. 
\begin{example}
We consider the group ring $\F_2G$ over the group $$G=\GEN{a,b\mid a^{11}=1,b^5=1,ba=a^3b}.$$
\begin{lstlisting}[frame=trbl]
gap> A:=FreeGroup("a","b");;a:=A.1;;b:=A.2;;
gap> G:=A/[a^11,b^5,b*a*b^(-1)*a^(-3)];;
gap> F:=GF(2);;         
gap> FG:=GroupRing(F,G);;
gap> S:=AsSet(G);;
gap> H:=StrongShodaPairs(G)[3][1];;
gap> K:=StrongShodaPairs(G)[3][2];; 
gap> N:=Normalizer(G,K);;
gap> epi:=NaturalHomomorphismByNormalSubgroup(N,K);;
gap> QHK:=Image(epi,H);;
gap> gq:=MinimalGeneratingSet(QHK)[1];;
gap> C:=CyclotomicClasses(Size(F),Index(H,K))[2];;
gap> P:=PrimitiveIdempotentsTrivialTwisting(FG,H,K,C,[epi,gq]);;
gap> e:=P[3];;   
gap> Ge := List(G,g->g*e);; time;
2156
gap> B := List(Ge,x->CodeWordByGroupRingElement(F,S,x));;
gap> code := GeneratorMatCode(B,F);
a linear [55,10,1..20]16..27 code defined by generator matrix over GF(2)
gap> MinimumDistance(code);
20
gap> LowerBoundMinimumDistance(55,10,2);
23
gap> UpperBoundMinimumDistance(55,10,2);
24
gap> DisplayProfile();
  count  self/ms  chld/ms  stor/kb  chld/kb  package  function              
      2      332        0    26433        0  (oprt.)  StrongShodaPairs      
     55    11680        0  3371416        0  (oprt.)  CodeWordByGroupRingElement
      1    59568        0  179861*        0  (oprt.)  PrimitiveIdempotentsTrivial*
               8               195                    OTHER                 
           71588           213842*                    TOTAL
\end{lstlisting}
In this way we have constructed a $[55,10,20]$-code by means of a minimal left metacyclic group code. By a computer search we were able to check that this code cannot be realized as an abelian group code. 
\end{example}

\section{Conclusions}

We list a table of minimal left group codes with best known minimal distance. The second column displays the group identification number in GAP. The last column displays the time in milliseconds needed to compute the code using our implementation.

\begin{center}
\begin{tabular}{|ccccc|}
 \hline 
$\F$ & $G$ & $k$ & $d_{\mbox{min}}$ & time\\
\hline
\verb+GF(2)+ & \verb+[ 21, 1 ]+ & 3 & 12 & 264 \\

\verb+GF(2)+ & \verb+[ 63, 1 ]+ & 3 & 36 & 300\\

\verb+GF(2)+ & \verb+[ 105, 1 ]+ & 3 & 60 & 520\\

\verb+GF(3)+ & \verb+[ 8, 4 ]+ & 2 & 6 & 144\\

\verb+GF(3)+ & \verb+[ 16, 4 ]+ & 2 & 12 & 268\\

\verb+GF(3)+ & \verb+[ 20, 3 ]+ & 4 & 12 & 280\\

\verb+GF(3)+ & \verb+[ 32, 2 ]+& 2 & 24 & 528\\

\verb+GF(3)+ & \verb+[ 40, 11 ]+ & 2 & 30 & 352\\

\verb+GF(3)+ & \verb+[ 56, 10 ]+ & 2 & 42 & 384\\

\verb+GF(3)+ & \verb+[ 64, 6 ]+ & 2 & 48 & 780\\

\verb+GF(3)+ & \verb+[ 80, 6 ]+ & 2 & 60 & 800\\
\hline
\end{tabular}
\end{center}

The next table contains minimal left group codes which cannot be realized as abelian group codes. The fifth column displays the maximum minimal distance achieved as can be found in \cite{1998Brouwer} or \cite{2007Grassl}.

\begin{center}
\begin{tabular}{|cccccc|}
 \hline
$\F$ & $G$ & $k$ & $d_{\mbox{min}}$ & best $d_{\mbox{min}}$ & time\\
\hline

\verb+GF(2)+ & \verb+[ 39, 1 ]+ & 12 & 6 & 14 & 604\\

\verb+GF(2)+ & \verb+[ 55, 1 ]+ & 10 & 20 & 23 & 700\\

\verb+GF(2)+ & \verb+[ 105, 1 ]+ & 12 & 36 & 44 & 1012\\

\verb+GF(3)+ & \verb+[ 20, 3 ]+ & 4 & 8 & 12 & 304\\

\verb+GF(3)+ & \verb+[ 40, 3 ]+ & 4 & 16 & 27 & 472\\

\verb+GF(4)+ & \verb+[ 39, 1 ]+ & 6 & 24 & 25 & 304\\

\verb+GF(4)+ & \verb+[ 55, 1 ]+ & 5 & 35 & 39 & 496\\

\verb+GF(5)+ & \verb+[ 21, 1 ]+ & 6 & 8 & 12 & 300\\

\hline
\end{tabular}
\end{center}

As answer to a question from \cite{2009BernalRioSimon}, Garc\'ia Pillado et.al. \cite{Pillado2013} constructed a two-sided group code over $\F_5$ which is not an abelian group code. More specifically, this code was realized by the group $S_4$ and is a $[24,9,8]$-code. They also proved that over $\F_5$ this code has the smallest possible length among all non-abelian group codes. This two-sided example of \cite{Pillado2013} can also be found using the construction of primitive central idempotents in Wedderga.

For left group codes, such a minimal length is not known. However, over $\F_5$, we constructed the left group code $[21,6,8]$ which is not an abelian group code.

We were able to achieve the desired goal of finding some optimal codes and non-abelian left group codes among the minimal left group codes. However all optimal codes found are well known and have small dimensions. This is due to the facts that we only considered minimal codes coming from semisimple group algebras. We also had computer memory limitations for searching through bigger groups, i.e. codes of higher length. Furthermore we were limited by the internal description of big fields in GAP. 
For groups up to order about 100 our methods to compute codes seem to be quite efficient, however testing if they can be realized by abelian groups is highly time consuming.
Still there is the hope to obtain many optimal codes with left group code structure and more left group codes which are not abelian group codes.

\renewcommand{\bibname}{References}
\bibliographystyle{amsalpha}
\bibliography{references}

\newcommand{\etalchar}[1]{$^{#1}$}
\providecommand{\bysame}{\leavevmode\hbox to3em{\hrulefill}\thinspace}
\providecommand{\MR}{\relax\ifhmode\unskip\space\fi MR }
\providecommand{\MRhref}[2]{%
  \href{http://www.ams.org/mathscinet-getitem?mr=#1}{#2}
}
\providecommand{\href}[2]{#2}
\begin{thebibliography}{JdROVG13}

\bibitem[Art73]{1973Artin}
E.~Artin, \emph{Galoissche {T}heorie}, Verlag Harri Deutsch, Zurich, 1973.

\bibitem[BBC{\etalchar{+}}12]{GUAVA}
R.~Baart, T.~Boothby, J.~Cramwinckel, J.~Fields, D.~Joyner, R.~Miller,
  E.~Minkes, E.~Roijackers, L.~Ruscio, and C.~Tjhai, \emph{{\textit{GUAVA - a
  GAP package}}}, Version 3.12, 21/05/2012,
  {\url{http://www.gap-system.org/Packages/guava.html}}.

\bibitem[BdR07]{Broche2007}
O.~Broche and {\'A}.~del R\'io, \emph{Wedderburn decomposition of finite group
  algebras}, Finite Fields Appl. \textbf{13} (2007), no.~1, 71--79.

\bibitem[BdRS09]{2009BernalRioSimon}
J.J. Bernal, {\'A}.~del R{\'{\i}}o, and J.J. Sim{\'o}n, \emph{An intrinsical
  description of group codes}, Des. Codes Cryptogr. \textbf{51} (2009), no.~3,
  289--300.

\bibitem[BHK{\etalchar{+}}13]{Wedderga}
O.~Broche, A.~Herman, A.~Konovalov, A.~Olivieri, G.~Olteanu, {\'A}.~del
  R{\'i}o, and I.~{Van Gelder}, \emph{{\textit{Wedderga - Wedderburn
  Decomposition of Group Algebras}}}, Version 4.6.0, 2013,
  {\url{http://www.cs.st-andrews.ac.uk/~alexk/wedderga},
  \url{http://www.gap-system.org/Packages/wedderga.html}}.

\bibitem[Bro98]{1998Brouwer}
A.E. Brouwer, \emph{{Bounds on the size of linear codes}}, Handbook of Coding
  Theory (Vera~S. Pless and W.Cary Huffman, eds.), Elsevier, Amsterdam, 1998,
  pp.~295--461.

\bibitem[CS89]{1989ChengSloane}
Y.~Cheng and N.J.A. Sloane, \emph{Codes from symmetry groups and a [32,17,8]
  code}, SIAM J. Disc. Math \textbf{2} (1989), 28--37.

\bibitem[Gao93]{1993Gao}
S.~Gao, \emph{Normal bases over finite fields}, Ph.D. thesis, University of
  Waterloo, Waterloo, 1993.

\bibitem[GAP13]{GAP}
The GAP~Group, \emph{{GAP -- Groups, Algorithms, and Programming, Version
  4.6.5}}, 2013, \url{http://www.gap-system.org}.

\bibitem[GPGM{\etalchar{+}}13]{Pillado2013}
C.~Garc\'ia~Pillado, S.~Gonz\'alez, C.~Martin\'ez, V.~Markov, and A.~Nechaev,
  \emph{Group codes over non-abelian groups}, J. Algebra Appl. \textbf{12}
  (2013), no.~7, 1350037 (20 pages).

\bibitem[Gra06]{2006Grassl}
M.~Grassl, \emph{{Searching for linear codes with large minimum distance}},
  Discovering Mathematics with Magma --- Reducing the Abstract to the Concrete
  (Wieb Bosma and John Cannon, eds.), Algorithms and Computation in
  Mathematics, vol.~19, Springer, Heidelberg, 2006, pp.~287--313.

\bibitem[Gra07]{2007Grassl}
\bysame, \emph{{Bounds on the minimum distance of linear codes and quantum
  codes}}, Online available at \url{http://www.codetables.de}, 2007, Accessed
  on 2013-09-15.

\bibitem[JdROVG13]{2013JdROVG}
E.~Jespers, {\'A}.~del R{\'i}o, G.~Olteanu, and I.~Van~Gelder, \emph{Group
  rings of finite strongly monomial groups: Central units and primitive
  idempotents}, J. Algebra \textbf{387} (2013), 99--116.

\bibitem[JLP03]{Jespers2003}
E.~Jespers, G.~Leal, and A.~Paques, \emph{Central idempotents in the rational
  group algebra of a finite nilpotent group}, J. Algebra Appl. \textbf{2}
  (2003), no.~1, 57--62.

\bibitem[Len91]{1991Lenstra}
H.~W. Lenstra, \emph{Finding isomorphisms between finite fields}, Math. Comp.
  \textbf{56} (1991), no.~193, 329--347.

\bibitem[L{\"u}n86]{1985Luneburg}
H.~L{\"u}neburg, \emph{On a little but useful algorithm}, Proceedings of the
  3rd International Conference on Algebraic Algorithms and Error-Correcting
  Codes (London, UK), Springer-Verlag, 1986, pp.~296--301.

\bibitem[OdRS04]{Olivieri2004}
A.~Olivieri, {\'A}.~del R{\'i}o, and J.J. Sim{\'o}n, \emph{On monomial
  characters and central idempotents of rational group algebras}, Comm. Algebra
  \textbf{32} (2004), no.~4, 1531--1550.

\bibitem[OVG11]{2011vangelder}
G.~Olteanu and I.~Van~Gelder, \emph{Finite group algebras of nilpotent groups:
  A complete set of orthogonal primitive idempotents}, Finite Fields Appl.
  \textbf{17} (2011), no.~2, 157--165.

\bibitem[Pas89]{Passman1989}
D.S. Passman, \emph{Infinite crossed products}, Pure Appl. Math., vol. 135,
  Academic Press, Boston, 1989.

\bibitem[Rei75]{Reiner1975}
I.~Reiner, \emph{Maximal orders}, Academic Press, London, New York, San
  Fransisco, 1975.

\bibitem[Rom06]{Roman2006}
S.~Roman, \emph{Field theory}, Graduate Texts in Mathematics, vol. 158,
  Springer, New York, 2006.

\bibitem[SL95]{sabin}
R.E. Sabin and S.J. Lomonaco, \emph{Metacyclic error-correcting codes}, Appl.
  Algebra Engrg. Comm. Comput. \textbf{6} (1995), no.~3, 191--210.

\bibitem[Yam73]{Yamada1973}
T.~Yamada, \emph{The {S}chur subgroup of the {B}rauer group}, Lect. Notes Math,
  vol. 397, Springer-Verlag, 1973.

\end{thebibliography}

\end{document}